\documentclass{article}

\usepackage[english]{babel}
\usepackage[utf8]{inputenc}
\usepackage[T1]{fontenc}

\usepackage{amsmath, amsthm, amssymb, mathtools}

\usepackage[shortlabels]{enumitem}

\usepackage{cases}

\usepackage{sectsty}

\usepackage{tikz}
\usetikzlibrary{positioning}
\usepackage{caption}

\usepackage[a4paper]{geometry}

\usepackage[displaymath, mathlines]{lineno}
\AtBeginDocument{\def\MR#1{}}

\usepackage{hyperref}

\providecommand{\keywords}[1]{\textbf{Keywords.} #1}
\providecommand{\MSC}[1]{\textbf{2010 Mathematics Subject Classification.} #1}

\newtheorem{theorem}{Theorem}[section]

\newtheorem{lemma}[theorem]{Lemma}

\newtheorem{example}[theorem]{Example}

\theoremstyle{definition}
\newtheorem{definition}[theorem]{Definition}
\newtheorem{remark}[theorem]{Remark}

\renewcommand\epsilon{\varepsilon}
\renewcommand\mapsto{\longmapsto}

\usepackage{color}

\usepackage{xspace}

\newcommand{\R}{\field{R}\xspace}

\newcommand{\N}{\field{N}\xspace}

\newcommand{\field}[1]{\ensuremath{\mathbb{#1}}}

\newcommand{\ens}[1]{ \left\{#1\right\} }

\newcommand\diag{\mathrm{diag} \,}

\newcommand{\Tinf}{T_{\mathrm{inf}}}

\newcommand\pt[1]{\frac{\partial #1}{\partial t}}
\newcommand\px[1]{\frac{\partial #1}{\partial x}}

\newcommand\yin{y^{\mathrm{in}}}
\newcommand\yout{y^{\mathrm{out}}}
\newcommand\uin{u^{\mathrm{in}}}

\newcommand{\rank}{\mathrm{rank} \,}

\def\norm#1{\left\|#1\right\|}
\newcommand\abs[1]{\left|#1\right|}
\newcommand\st{\quad \middle| \quad}

\newcommand\supp{\mathrm{supp} \,}

\usepackage{ifthen}
\newcommand{\syst}[2]{
\ifthenelse{\equal{#2}{}}{\left(\Lambda,#1,Q\right)}
{\ifthenelse{\equal{#2}{b}}{\left(\Lambda,-,Q,#1\right)}{}}
{\ifthenelse{\equal{#2}{c}}{\left(\Lambda,-,Q^0,#1\right)}{}}
}

\newcommand\clos[1]{\overline{#1}}

\newcommand\taumax{T^*}
\newcommand\taumaxpeps{T^{**}}
\newcommand\omegahat{{\omega_0}}

\newcommand{\rev}[1]{\widehat{#1}}
\newcommand\concompo{\mathcal{C}}

\newcommand\modif[1]{{{#1}}}
\newcommand{\Tinfbc}{T_{\mathrm{inf}}^{\mathrm{bc}}}

\title{The minimal control time for the exact controllability by internal controls of 1D linear hyperbolic balance laws}

\author{
Long Hu\thanks{School of Mathematics, Shandong University, Jinan, Shandong 250100, China.  E-mail: \texttt{hul@sdu.edu.cn}}
\and
Guillaume Olive\thanks{Faculty of Mathematics and Computer Science, Jagiellonian University, ul. {\L}ojasiewicza 6, 30-348 Krak\'{o}w, Poland. E-mail: \texttt{math.golive@gmail.com} or \texttt{guillaume.olive@uj.edu.pl}}
}

\date{\today}

\begin{document}

\maketitle

\begin{abstract}
In this article we study the internal controllability of 1D linear hyperbolic balance laws when the number of controls is equal to the number of state variables.
The controls are supported in space in an arbitrary open subset.
Our main result is a complete characterization of the minimal control time for the exact controllability property.
\end{abstract}

\keywords{Hyperbolic systems; Minimal control time; Internal controllability}

\vspace{0.2cm}
\MSC{35L40; 93B05}

\tableofcontents

%\clearpage
\section{Introduction and main result}

\subsection{Problem description}\label{sect prob desc}

In this paper, we are interested in the controllability properties of a class of one-dimensional (1D) first-order linear hyperbolic systems (see e.g. \cite{BC16} for applications).
The equations of the system are
\begin{subequations}\label{syst}
\begin{equation}\label{syst:equ}
\pt{y}(t,x)+\Lambda(x) \px{y}(t,x)=M(x) y(t,x) +1_\omega(x) u(t,x).
\end{equation}
Above, $t \in (0,T)$ is the time variable, $T>0$, $x \in (0,1)$ is the space variable and the state is $y:(0,T) \times (0,1) \to \R^n$ $(n \geq 2$).
The matrix $\Lambda \in C^{0,1}([0,1])^{n \times n}$ will always be assumed diagonal $\Lambda =\diag(\lambda_1,\ldots,\lambda_n)$, with $m \geq 1$ negative speeds and $p \geq 1$ positive speeds ($m+p=n$) such that:
$$
\lambda_1(x) \leq \cdots \leq \lambda_m(x) <0<\lambda_{m+1}(x) \leq \cdots \leq \lambda_{m+p}(x), \quad \forall x \in [0,1],
$$
and
$$(\lambda_k(x)=\lambda_l(x) \quad \text{ for some } x \in [0,1]) \quad \Longrightarrow \quad (\lambda_k(x)=\lambda_l(x) \quad \text{ for every } x \in [0,1]),$$
for every $k,l \in \ens{1,\ldots,n}$.
The matrix $M \in L^{\infty}(0,1)^{n \times n}$ couples the equations of the system inside the domain.
The function $u:(0,T)\times(0,1) \to \R^n$ is called the control, it will be at our disposal.
It is crucial to point out that, all along this work, the control has the same number of components as the state.
On the other hand, it is allowed to act only on a subdomain $(0,T) \times \omega$, where $\omega$ is a fixed but arbitrary nonempty open subset of $(0,1)$.

The system will be evolving forward in time, so we consider an initial condition at time $t=0$.

Let us now discuss the boundary conditions.
The structure of $\Lambda$ induces a natural splitting of the state into components corresponding to negative and positive speeds, denoted respectively by $y_-$ and $y_+$.
For the above system to be well-posed, we then need to add boundary conditions at $x=1$ for $y_-$ and at $x=0$ for $y_+$.
We will consider the following type of boundary conditions:
\begin{equation}\label{syst:BC}
y_-(t,1)=Q_1 y_+(t,1), \quad y_+(t,0)=Q_0 y_-(t,0),
\end{equation}
where $Q_0 \in \R^{p \times m}$ and $Q_1 \in \R^{m \times p}$ will be called the boundary coupling matrices.
\end{subequations}

In what follows, \eqref{syst:equ} and \eqref{syst:BC} with an initial condition $y(0,x)=y^0(x)$ will be referred to as system \eqref{syst}.
We recall that it is well-posed in the following functional setting: for every $T>0$, $y^0 \in L^2(0,1)^n$ and $u \in L^2(0,T;L^2(0,1)^n)$, there exists a unique solution $y \in C^0([0,T];L^2(0,1)^n)$.
By solution we mean weak solution and we refer for instance \cite[Appendix A]{BC16} for a proof of this well-posedness result.

The regularity $C^0([0,T];L^2(0,1)^n)$ of the solution allows us to consider control problems in the space $L^2(0,1)^n$:

\begin{definition}
Let $T>0$ be fixed.
We say that system \eqref{syst} is exactly controllable in time $T$ if, for every $y^0, y^1 \in L^2(0,1)^n$, there exists $u \in L^2(0,T;L^2(0,1)^n)$ such that the corresponding solution $y$ to system \eqref{syst} satisfies $y(T,\cdot)=y^1$.
\end{definition}

Since controllability in time $T_1$ implies controllability in any time $T_2 \geq T_1$, it is natural to try to find the smallest possible control time, the so-called ``minimal control time''.

\begin{definition}
We denote by $\Tinf \in [0,+\infty]$ the minimal control time (for the exact controllability) of system \eqref{syst}, that is
$$
\Tinf=
\inf\ens{T>0 \st \text{System \eqref{syst} is exactly controllable in time $T$}}.
$$
\end{definition}

The time $\Tinf$ is named ``minimal'' control time according to the current literature, despite it is not always a minimal element of the set.
We keep this naming here, but we use the notation with the ``inf'' to avoid eventual confusions.
The goal of this article is to characterize $\Tinf$.

In order to state our result and those of the literature, we need to introduce the following times:
\begin{equation}\label{def TkI}
T_k^I=\int_I \frac{1}{\abs{\lambda_k(\xi)}} \, d\xi, \quad 1 \leq k \leq n,
\end{equation}
\modif{for any nonempty interval $I \subset (0,1)$.}
We will also use the simpler notation $T_k$ when $I=(0,1)$.
The time $T_k$ is the time needed for the boundary controllability of a single equation (the transport equation) with speed $\lambda_k$.
Note that the assumption on the speeds implies in particular the following order relation (for any $I$):
$$
T_1^I \leq \cdots \leq T_m^I \quad \text{ and } \quad T_n^I \leq \cdots \leq T_{m+1}^I.
$$

Finally, we recall the key notion introduced in \cite[Section 1.2]{HO22-JDE} of ``canonical form'' for boundary coupling matrices.

\begin{definition}
We say that a matrix $Q \in \R^{k \times l}$ ($k,l \geq 1$ arbitrary) is in canonical form if it has at most one nonzero entry on each row and each column, and this entry is equal to $1$.
We denote by $(r_1,c_1), \ldots, (r_\rho,c_\rho)$ the positions of the corresponding nonzero entries, with $r_1<\ldots<r_\rho$.
\end{definition}

We can prove that, for every $Q \in \R^{k \times l}$, there exists a unique $Q^0 \in \R^{k \times l}$ in canonical form such that $LQU=Q^0$ for some upper triangular matrix $U \in \R^{l \times l}$ with only ones on its diagonal and some invertible lower triangular matrix $L \in \R^{k \times k}$.
The matrix $Q^0$ is called the canonical form of $Q$ and we can extend the definition of $(r_1,c_1), \ldots, (r_\rho,c_\rho)$ to any nonzero matrix.
Obviously, when $Q$ is full row rank (i.e. $\rank Q=k$), we have $r_\alpha=\alpha$ for every $\alpha \in \ens{1,\ldots,k}$.
We refer to the above reference for more details.

\subsection{Literature}\label{sect literature}

The controllability of 1D first-order hyperbolic systems with boundary controls (see e.g. system \eqref{syst bound cont} below) has been widely studied.
Two pioneering works are \cite{Cir69} for quasilinear systems and \cite{Rus78} for linear systems.
In \cite{Rus78} the author raised the open problem of finding the minimal control time.
This was solved in the particular case $M=0$ few years later in \cite{Wec82} for the null controllability property (i.e. when we only want to reach $y^1=0$).
Afterwards, it seems that the community turned its attention to quasilinear systems in the $C^1$ framework of so-called semi-global solutions (see e.g. \cite{LR03,Li10,Hu15}).
However, lately there has been a resurgence on finding the minimal control time.
This was initiated by the authors in \cite{CN19} and followed by a series of works \cite{CN21,CN21-pre,HO21-JMPA,HO21-COCV,HO22-JDE,HO23-pre}.
Notably, in \cite{HO21-JMPA} we gave a complete answer for the exact controllability property.
Given the importance of this result for the present work, we recall it in detail:

\begin{theorem}\label{thm HO21-JMPA}
\modif{Let $0<a \leq 1$ be fixed.
Let $\tau_-(0,a) \in [0,+\infty]$} denote the minimal control time for the exact controllability of the system 
\begin{equation}\label{syst bound cont}
\begin{dcases}
\pt{y}(t,x)+\Lambda(x) \px{y}(t,x)=M(x) y(t,x), \\
y_-(t,a)=u(t), \quad y_+(t,0)=Q_0 y_-(t,0), \\
y(0,x)=y^0(x),
\end{dcases}
\quad t \in (0,T), \, x \in \modif{(0,a),}
\end{equation}
in which $u \in L^2(0,T;\R^m)$ is the control.
Then, we have:
\begin{enumerate}
\item
\modif{$\tau_-(0,a)<+\infty$} if, and only if, $\rank Q_0=p$.

\item
If $\rank Q_0=p$, then
\modif{
\begin{equation}\label{Tinfbc0a}
\tau_-(0,a)=\max_{1 \leq j \leq p} \ens{T_{m+j}^{(0,a)}+T_{c_j}^{(0,a)}, \quad T_m^{(0,a)}},
\end{equation}
where the indices $c_j$ refer to $Q_0$.
}
\end{enumerate}
\end{theorem}

\modif{
\begin{remark}\label{rem cont at x=0}
The case of a control acting on the other part of the boundary can be reduced to the previous one after changing $x$ to $1-x$ and then relabeling the unknowns backwards (i.e. considering $\hat{y}_k(t,x)=y_{n+1-k}(t,1-x)$).
More precisely, let $0 \leq b<1$ be fixed, and let $\tau_+(b,1) \in [0,+\infty]$ denote the minimal control time for the exact controllability of the system 
$$
\begin{dcases}
\pt{y}(t,x)+\Lambda(x) \px{y}(t,x)=M(x) y(t,x), \\
y_-(t,1)=Q_1 y_+(t,1), \quad y_+(t,b)=v(t), \\
y(0,x)=y^0(x),
\end{dcases}
\quad t \in (0,T), \, x \in (b,1),
$$
in which $v \in L^2(0,T;\R^p)$ is the control.
Then, $\tau_+(b,1)<+\infty$ if, and only if, $\rank Q_1=m$, and in that case we have
\begin{equation}\label{Tinfbcb1}
\tau_+(b,1)=\max_{1 \leq i \leq m} \ens{T_{m+1-i}^{(b,1)} +T_{n+1-c_i}^{(b,1)}, \quad T_{m+1}^{(b,1)}},
\end{equation}
where the indices $c_i$ refer to the matrix $\rev{Q}_0$ whose $(i,j)$ entry is the $(m+1-i,p+1-j)$ entry of $Q_1$.
\end{remark}
}

Controllability of system \eqref{syst bound cont} is also sometimes referred to as one-sided controllability, because the control acts only on one side of the boundary.
The situation where two controls are used, or two-sided controllability, was also studied.
Notably, we can extract the following result from the literature (see \cite[Remark 1.10 (ii)]{HO21-COCV}\footnote{To apply the backstepping method mentioned therein in our non-strictly hyperbolic framework, we first proceed as described in the first item of \cite[Remark 4.6]{HO21-JMPA}.}):

\modif{
\begin{theorem}\label{thm 2 controls}
Let $0 \leq c<d \leq 1$ be fixed.
Let $\tau_{-,+}(c,d)$ denote the minimal control time for the exact controllability of the system 
$$
\begin{dcases}
\pt{y}(t,x)+\Lambda(x) \px{y}(t,x)=M(x) y(t,x), \\
y_-(t,d)=u(t), \quad y_+(t,c)=v(t), \\
y(0,x)=y^0(x),
\end{dcases}
\quad t \in (0,T), \, x \in (c,d),
$$
in which $u \in L^2(0,T;\R^m)$ and $v \in L^2(0,T;\R^p)$ are the controls.
Then,
\begin{equation}\label{Tinfbccd}
\tau_{-,+}(c,d)=\max\ens{T_{m}^{(c,d)}, \quad T_{m+1}^{(c,d)}}.
\end{equation}
\end{theorem}
}

On the contrary, there are very few results concerning the internal controllability of 1D first-order hyperbolic systems.
The first ones seem \cite{ZLR16} and \cite{ABCO17}.
In \cite{ZLR16}, the authors studied the local exact controllability of quasilinear versions of system \eqref{syst} (for semi-global $C^1$ solutions).
This problem was also revisited in a linear setting in the recent paper \cite{LLQ24}.
In both of these articles, the authors proved in particular that system \eqref{syst} with an interval $\omega=(a,b)$ is exactly controllable in any time
\begin{equation}\label{cont time LLQ24}
T>\max\ens{a,1-b}(T_{m+1}+T_m),
\end{equation}
provided also that $Q_0,Q_1$ are invertible square matrices.
On the other hand, in \cite{ABCO17}, the authors investigated the more challenging problem of internal controllability by a reduced number of controls.
For a system of two equations with periodic boundary conditions, they proved the local exact controllability by only one control.
The present paper is about linear systems when the number of controls is equal to the number of state variables, in the same spirit as in \cite{LLQ24}.
We will generalize the result mentioned above by obtaining the best control time (i.e. $\Tinf$), by considering an arbitrary open set $\omega$ and by showing the necessity of the assumptions on $Q_0,Q_1$.

\subsection{Main result and comments}

The main result of this paper is the following complete characterization of the exact controllability properties for system \eqref{syst}.

\begin{theorem}\label{main thm}
Let $\omega \subset (0,1)$ be a nonempty open subset.
\begin{enumerate}
\item\label{main thm i1}
Assume that
\begin{equation}\label{rank Q}
Q_0,Q_1 \text{ are invertible, }
\end{equation}
(necessarily, $m=p=n/2$).
\modif{Then, the minimal control time $\Tinf$ of system \eqref{syst} is
\begin{equation}\label{min cont time}
\Tinf=
\begin{dcases}
0 & \text{ if } \clos{\omega}=[0,1], \\
\max_{I \in \concompo(\clos{\omega}^c)} \Tinfbc(I) & \text{ otherwise, }
\end{dcases}
\end{equation}
where:
\begin{itemize}
\item
$\concompo(\clos{\omega}^c)$ denotes the set of connected components of $\clos{\omega}^c=(\R \setminus \clos{\omega}) \cap (0,1)$.

\item
$\Tinfbc(I)$ denotes the minimal control time for the exact controllability of the system
\begin{equation}\label{syst I:equ}
\begin{dcases}
\pt{y}(t,x)+\Lambda(x) \px{y}(t,x)=M(x) y(t,x), \\
y(0,x)=y^0(x),
\end{dcases}
\quad t \in (0,T), \, x \in I,
\end{equation}
with the following boundary conditions:
\begin{equation}\label{syst I:BC}
\begin{aligned}
& \text{if $I=(0,a)$ with $0<a<1$:}
& y_-(t,a) &=u(t), & y_+(t,0) &=Q_0 y_-(t,0),
\\
& \text{if $I=(b,1)$ with $0<b<1$:}
& y_-(t,1) &=Q_1y_+(t,1), & y_+(t,b) &=v(t),
\\
& \text{if $I=(c,d)$ with $0<c<d<1$:} 
& y_-(t,d) &=u(t), & y_+(t,c) &=v(t),
\end{aligned}
\end{equation}
(above, $u,v \in L^2(0,T;\R^\frac{n}{2})$ are the controls).
We recall that $\Tinfbc(I)$ is explicitly given by one of the formulas \eqref{Tinfbc0a}, \eqref{Tinfbcb1} or \eqref{Tinfbccd}.
\end{itemize}
}

\item\label{main thm NC}
Conversely, if system \eqref{syst} is exactly controllable in some time, then \eqref{rank Q} holds.
\end{enumerate}
\end{theorem}

We recall that any open set $U$ in $\R$ can be partitioned into disjoint open intervals $I_k$, $k \in S$, with countable $S$, and the connected components of $U$ are simply these $I_k$.
If $U$ is bounded and $S$ is infinite, then necessarily $\abs{I_k} \to 0$.
It follows \modif{from the explicit formulas that $\Tinfbc(I_k) \to 0$} as well, and the notation \modif{$\max \Tinfbc$} used in the statement of our result is thus justified.

\begin{remark}
Theorem \ref{main thm} generalizes \cite[Theorem 2.1]{LLQ24} (and, to some extent, also \cite[Theorem 5.1]{ZLR16}) on important aspects:
\begin{itemize}
\item
We obtain the best possible control time, \modif{with an explicit way to compute it (see e.g. Example \ref{examples} below).}

\item
We consider an arbitrary open subset $\omega$, not just an interval.

\item
We show that the invertibility of the matrices $Q_0,Q_1$ are in fact necessary.
\end{itemize}
\end{remark}

\begin{example}\label{examples}
Assume that $\Lambda$ is constant \modif{(in that case, we simply have $T_k^I=\abs{I} T_k$, see \eqref{def TkI}).}
\begin{enumerate}
\item
Consider an interval $\omega=(a,b)$ and assume that $Q_0,Q_1$ are the identity matrix.
Then, \modif{combining our main result with Theorem \ref{thm HO21-JMPA} (and Remark \ref{rem cont at x=0}), we have}
$$
\Tinf=\max\ens{a,1-b} \max_{1 \leq i \leq \frac{n}{2}} \ens{T_{\frac{n}{2}+i}+T_i}.
$$
\modif{Note that, unless $(a,b)=(0,1)$,} this time is strictly smaller than the control time given in \cite{LLQ24} (see \eqref{cont time LLQ24}).

\item
\modif{Assume that $\omega$ touches both parts of the boundary (i.e. $0,1 \in \partial\omega$), $\clos{\omega} \neq [0,1]$, and that $Q_0,Q_1$ are invertible.
Then, combining our main result with Theorem \ref{thm 2 controls}, we see that
$$
\Tinf=L\max\ens{T_{\frac{n}{2}}, \quad T_{\frac{n}{2}+1}},
$$
where $L=\max_{I \in \concompo(\clos{\omega}^c)} \abs{I}$ belongs to $(0,1)$.
In particular, system \eqref{syst} is exactly controllable in any time $T \geq \max\ens{T_{\frac{n}{2}}, \quad T_{\frac{n}{2}+1}}$.}

\end{enumerate}
\end{example}

\begin{remark}
Let us also mention the work \cite{BO14}, where it is shown that the connected components of the complement of the control domain can also play an important role in the controllability properties of 1D parabolic systems.
\end{remark}

\section{The minimal control time}\label{sect suff cond}

\modif{
In this part, we prove the first statement of our main result.
Let us denote by $\taumax$ the right-hand side of \eqref{min cont time}, that is
\begin{equation}\label{def taumax}
\taumax=
\begin{dcases}
0 & \text{ if } \clos{\omega}=[0,1], \\
\max_{I \in \concompo(\clos{\omega}^c)} \Tinfbc(I) & \text{ otherwise. }
\end{dcases}
\end{equation}

We first point out that the minimal control time by boundary controls $\Tinfbc$ enjoys the following important properties, as is clear from the explicit formulas \eqref{Tinfbc0a}, \eqref{Tinfbcb1} and \eqref{Tinfbccd}:
\begin{enumerate}[(a)]
\item\label{tau prop 1}
$\Tinfbc(I) \leq C\abs{I}$ for every $I$, for some some $C>0$ independent of $I$.

\item\label{tau prop 2}
$\Tinfbc(I_k) \to \Tinfbc(I)$ for every $I$ and every sequence $I_k \supset I$ such that $\abs{I_k \setminus I} \to 0$.

\item\label{tau prop 3}
$\Tinfbc(I_1) \leq \Tinfbc(I_2)$ for every $I_1 \subset I_2$.
\end{enumerate}

Based on these three properties, we will prove the following essential lemma.

\begin{lemma}\label{lem omegahat}
For every $\epsilon>0$, there exists a nonempty open subset $\omegahat \subset\subset \omega$ such that:
\begin{enumerate}
\item
$\omegahat$ is a finite union of disjoint open intervals.

\item
$\max_{I \in \concompo(\clos{\omegahat}^c)} \Tinfbc(I)
\leq \taumax +\epsilon$.
\end{enumerate}

\end{lemma}

Let us admit this lemma for a moment and let us prove the first statement of our main result.
}

\begin{proof}[\modif{Proof of Theorem \ref{main thm}, item \ref{main thm i1}}]
\begin{enumerate}
\item
Naturally, we start with the case $\omega=(0,1)$.
This means that the control operator is simply the identity.
Now observe that system \eqref{syst} is time reversible since $Q_0,Q_1$ are assumed to be invertible.
It is easy and well-known how to control such a system in any time.
Indeed, let $T>0$ and $y^0,y^1 \in L^2(0,1)^n$ be fixed.
Let $\eta \in C^1([0,T])$ be a time cut-off function such that
$$\eta(0)=1, \quad \eta(T)=0.$$
We define
\begin{equation}\label{def y omega full}
y(t,x)=\eta(t)y^f(t,x)+(1-\eta(t)) y^b(t,x),
\end{equation}
where $y^f$ is the solution to the forward problem (without control)
$$
\begin{dcases}
\pt{y^f}(t,x)+\Lambda(x) \px{y^f}(t,x)=M(x) y^f(t,x), \\
y^f_-(t,1)=Q_1 y^f_+(t,1), \quad y^f_+(t,0)=Q_0 y^f_-(t,0), \\
y^f(0,x)=y^0(x),
\end{dcases}
\quad t \in (0,T), \, x \in (0,1),
$$
and $y^b$ is the solution to the backward problem (without control)
$$
\begin{dcases}
\pt{y^b}(t,x)+\Lambda(x) \px{y^b}(t,x)=M(x) y^b(t,x), \\
y^b_+(t,1)=Q_1^{-1} y^b_-(t,1), \quad y^b_-(t,0)=Q_0^{-1} y^b_+(t,0), \\
y^b(T,x)=y^1(x),
\end{dcases}
\quad t \in (0,T), \, x \in (0,1).
$$
Note that this second system is well-posed (it can be put in the form of the first one by considering the change of variable $t \mapsto T-t$ and then relabeling the unknowns).
By construction, $y$ satisfies the boundary conditions \eqref{syst:BC}, the initial condition $y(0,\cdot)=y^0$ and the final condition $y(T,\cdot)=y^1$.
The equations will also be satisfied if we take as control
\begin{equation}\label{def cont formal}
u(t,x)=\pt{y}(t,x)+\Lambda(x) \px{y}(t,x) -M(x) y(t,x).
\end{equation}
This expression is a priori a formal one, but it can be made rigorous by using the equations satisfied by $y^f$ and $f^b$ and defining in fact
$$u(t,x)=\eta'(t)(y^f(t,x)-y^b(t,x)).$$
Note that $u \in C^0([0,T];L^2(0,1)^n)$.
We can then check that $y$ defined by \eqref{def y omega full} is indeed the weak solution to system \eqref{syst} associated with this $u$.

\item
Let us now consider an arbitrary open subset $\omega \subset (0,1)$.
\modif{We want to show that $\Tinf=\taumax$, where $\taumax$ is given by \eqref{def taumax}.}
The inequality $\Tinf \geq \taumax$ is clear.
Indeed, if $T>\Tinf$ and $\clos{\omega} \neq [0,1]$, then, for any open interval $I \subset \clos{\omega}^c$, \modif{the system satisfied by the restriction of $y$ to $(0,T) \times I$, which is of the form \eqref{syst I:equ}-\eqref{syst I:BC} (with controls in $L^2(0,T;\R^\frac{n}{2})$ since $y \in C^0([0,1];L^2(0,T)^n)$), is also exactly controllable in time $T$.}
Let now $T$ be fixed such that $T>\taumax$ and let us prove that system \eqref{syst} is exactly controllable in time $T$.
\modif{By Lemma \ref{lem omegahat},} there exists a nonempty open subset $\omegahat \subset\subset \omega$ such that
\begin{equation}\label{time for omegahat}
\modif{\max_{I \in \concompo(\clos{\omegahat}^c)} \Tinfbc(I)} <T,
\end{equation}
and \modif{$\clos{\omegahat}^c$ is the union of} some disjoint open intervals $I_1,\ldots,I_N \subset (0,1)$.
We now define
\begin{equation}\label{def y}
y(t,x)=\xi(x) \yout(t,x) +(1-\xi(x)) \yin(t,x),
\end{equation}
where:
\begin{itemize}
\item
$\xi \in C^1([0,1])$ is a space cut-off function such that
$$
\xi(x)=
\begin{dcases}
1 & \text{ if } x \not\in \clos{\omega_1}, \\
0 & \text{ if } x \in \clos{\omegahat},
\end{dcases}
$$
where $\omega_1$ is an arbitrary open subset such that $\omegahat \subset\subset \omega_1 \subset\subset \omega$.

\item
$\yout=y^{I_k}$ in $[0,T] \times I_k$, where $y^{I_k} \in C^0([0,T];L^2(I_k)^n)$ is the controlled solution to \eqref{syst I:equ}-\eqref{syst I:BC} (with $I=I_k$) satisfying $y^{I_k}(T,x)=y^1(x)$ for $x \in I_k$ (whose existence is guaranteed by \eqref{time for omegahat}).
Outside the previous domains, we simply set $\yout=0$.

\item
$\yin \in C^0([0,T];L^2(0,1)^n)$ is the solution controlled over the whole domain constructed in the first step of the proof (i.e. $\yin$ is $y$ defined in \eqref{def y omega full}).
\end{itemize}

As before, $y$ satisfies the boundary conditions \eqref{syst:BC}, the initial condition $y(0,\cdot)=y^0$, the final condition $y(T,\cdot)=y^1$ and the control $u$ is then formally defined by \eqref{def cont formal}.
More rigorously,
$$u(t,x)=\xi'(x)\Lambda(x) (\yout(t,x)-\yin(t,x))+(1-\xi(x)) \uin(t,x),$$
where $\uin$ is the control associated with $\yin$.
Clearly, $\supp u \subset [0,T] \times \omega$.
Finally, we can check again that $y$ defined by \eqref{def y} is indeed the weak solution to system \eqref{syst} associated with this $u$.

\end{enumerate}
\end{proof}

\begin{remark}
The second step in the proof above is inspired by the proof of \cite[Theorem 2.2]{AKBGBdT11-MCRF} on the equivalence between internal and boundary controllability for the heat equation.
However, whereas in this parabolic setting the choice of $\omegahat$ has no influence on the controllability properties, it is not the case in our hyperbolic setting and we must be more subtle as this choice may affect the control time.
In the context of hyperbolic systems, it was used in the proof of \cite[Proposition 3.2]{ABCO17} (see also \cite{ZLR16}).
The proof of \cite[Theorem 2.1]{LLQ24} is also based on this idea, it is the construction of $\yout$ which is slightly different from \cite{ABCO17}, because of different boundary conditions.
\end{remark}

We conclude this section with the proof of the technical lemma.

\begin{proof}[Proof of Lemma \ref{lem omegahat}]
\begin{enumerate}
\item
Let us denote by \modif{$\taumaxpeps=\taumax+\epsilon/2$.}
Then $\taumaxpeps>0$ and, using Property \ref{tau prop 1}, we see that we can split the interval $[0,1]$ as follows:
$$
[0,1]=\bigcup_{k=0}^N [a_k,a_{k+1}], \quad \Tinfbc(a_k,a_{k+1}) \leq \taumaxpeps,
$$
for some partition $0=a_0<a_1<\ldots<a_N<a_{N+1}=1$.

\item
Let us denote by $K$ the set of indices $0 \leq k \leq N$ such that
$$
\omega \cap (a_k,a_{k+1}) \neq \emptyset.
$$
Obviously, $K$ is not empty since $\omega$ is not either.
For $k \in K$, we define
$$
a_k^+=\inf \omega \cap (a_k,a_{k+1}), \quad a_{k+1}^-=\sup \omega \cap (a_k,a_{k+1}).
$$
Note that $a_k^+<a_{k+1}^-$ since $\omega \cap (a_k,a_{k+1})$ is open.
In the sequel, we denote by $K=\ens{k_1,\ldots,k_r}$ with $k_1<\ldots<k_r$.

\item
By very definition, if $(0,a_{k_1}^+)$ is not empty, then it is a connected component of $\clos{\omega}^c$ and it follows that
$$\Tinfbc(0,a_{k_1}^+) \leq \taumax \leq \taumaxpeps.$$
By Property \ref{tau prop 2}, there exists $\delta_0>0$ such that $(0,a_{k_1}^+ +\delta_0) \subset (0,1)$ with
\begin{equation}\label{estim near 0}
\Tinfbc(0, a_{k_1}^+ +\delta_0) \leq \taumaxpeps +\frac{\epsilon}{2}.
\end{equation}
Similarly, if $(a_{k_r+1}^-,1)$ is not empty, then there exists $\delta_r>0$ such that $(a_{k_r+1}^- -\delta_r, 1) \subset (0,1)$ with
\begin{equation}\label{estim near 1}
\Tinfbc(a_{k_r+1}^- -\delta_r, 1) \leq \taumaxpeps +\frac{\epsilon}{2},
\end{equation}
and, if $r \geq 2$, $1 \leq l \leq r-1$ and $(a_{k_l+1}^-, a_{k_{l+1}}^+)$ is not empty, then there exists $\delta_l>0$ such that $(a_{k_l+1}^- -\delta_l, a_{k_{l+1}}^+ +\delta_l) \subset (0,1)$ with
\begin{equation}\label{estim near ak+1}
\Tinfbc(a_{k_l+1}^- -\delta_l, a_{k_{l+1}}^+ +\delta_l) \leq \taumaxpeps +\frac{\epsilon}{2}.
\end{equation}
Let us now define
$$
\delta=\min_{\substack{0 \leq l \leq r \\ 1 \leq \ell \leq r}} \ens{\delta_l, \frac{a_{k_\ell+1}^- - a_{k_\ell}^+}{2}},
$$
where it is understood that $\delta_l$ is simply omitted if it does not exist.
Then, all the previous properties \eqref{estim near 0}, \eqref{estim near 1} and \eqref{estim near ak+1} remain true with $\delta$ instead of $\delta_l$, thanks to Property \ref{tau prop 3}.

\item
For $k \in K$, by definition of $a_k^+$, we have
$$
\omega \cap (a_k^+, a_k^+ +\delta) \neq \emptyset.
$$
Therefore, there exists an open interval
$$J_k^+ \subset\subset \omega \cap (a_k^+, a_k^+ +\delta).$$
Similarly, there exists an open interval $J_{k+1}^- \subset\subset \omega \cap (a_{k+1}^--\delta, a_{k+1}^-)$.
Let us now define
$$
\omegahat=
\bigcup_{k \in K} J_k^+ \cup J_{k+1}^-.
$$

\item
By construction, the connected component $I$ of $\clos{\omegahat}^c$ located between $J_k^+$ and $J_{k+1}^-$ is included in $(a_k,a_{k+1})$ and thus, by Property \ref{tau prop 3} and definition of $a_k$, satisfies
$$\Tinfbc(I) \leq \Tinfbc(a_k,a_{k+1}) \leq \taumaxpeps.$$
On the other hand, the connected component $I$ of $\clos{\omegahat}^c$ located between $J_{k_l+1}^-$ and $J_{k_{l+1}}^+$ (when $r \geq 2$) is included in $(a_{k_l+1}^- -\delta,a_{k_{l+1}}^+ +\delta)$ and thus, by Property \ref{tau prop 3} and estimate \eqref{estim near ak+1}, satisfies
$$
\Tinfbc(I) \leq \Tinfbc(a_{k_l+1}^- -\delta,a_{k_{l+1}}^+ +\delta) \leq \taumaxpeps+\frac{\epsilon}{2}.
$$
Finally, for the connected component $I$ located before $J_{k_1}^+$, we distinguish two cases.
If $(0,a_{k_1}^+)$ is empty, then $I$ is included in $(0,a_1)$ and we obtain that $\Tinfbc(I) \leq \taumaxpeps$ as before.
If $(0,a_{k_1}^+)$ is not empty, then $I$ is included in $(0,a_{k_1}^+ +\delta)$ and we obtain that $\Tinfbc(I) \leq \taumaxpeps+\epsilon/2$ as before.
A similar argument applies to the connected component located after $J_{k_r+1}^-$.

\end{enumerate}

\end{proof}

\section{Invertibility of the boundary coupling matrices}\label{sect nec cond}

In this part, we prove that condition \eqref{rank Q} is necessary for system \eqref{syst} to be exactly controllable in some time.
To this end, it is equivalent to show that $Q_0,Q_1$ must both be of full row rank, i.e.
$$\rank Q_0=p, \quad \rank Q_1=m,$$
\modif{(which implies in particular that $p=m$).}

\begin{proof}[Proof of Theorem \ref{main thm}, item \ref{main thm NC}]
\begin{enumerate}
\item
Let us first make some preliminary observations that simplify the problem.
It is clear that we can assume that $\omega=(0,1)$ and take any particular $M$ since the right-hand side in system \eqref{syst} can then be considered as a new control.
We will consider $M=-\Lambda'$ because it simplifies the computations on the so-called adjoint system introduced below.
\modif{In other words, the controllability of system \eqref{syst} implies the controllability of the following system:}
\begin{equation}\label{syst part M}
\begin{dcases}
\pt{y}(t,x)+\Lambda(x) \px{y}(t,x)=-\Lambda'(x)y(t,x)+u(t,x), \\
y_-(t,1)=Q_1y_+(t,1), \quad y_+(t,0)=Q_0 y_-(t,0), \\
y(0,x)=y^0(x),
\end{dcases}
\quad t \in (0,T), \, x \in (0,1).
\end{equation}

It is also clear that we can consider $T$ as large as we want, and therefore assume at least that $T \geq \max\ens{T_1, T_n}$.

\item
We recall that, for system \eqref{syst part M} to be exactly controllable in time $T$, it is necessary (and sufficient) that the following so-called observability inequality holds (see e.g. \cite[Theorem 2.42]{Cor07}):
\begin{equation}\label{obs ineq}
\exists C>0, \quad \norm{z^1}_{L^2(0,1)^n}^2 \leq C \int_0^T \norm{z(t,\cdot)}_{L^2(0,1)^n}^2 \, dt, \quad \forall z^1 \in L^2(0,1)^n,
\end{equation}
where $z \in C^0([0,T];L^2(0,1)^n)$ is the solution to the adjoint system to \eqref{syst part M}, which in our case is
\begin{equation}\label{adj syst}
\begin{dcases}
\pt{z}(t,x)+\Lambda(x) \px{z}(t,x)=0, \\
z_-(t,0)=R_0^* z_+(t,0), \quad z_+(t,1)=R_1^* z_-(t,1), \\
z(T,x)=z^1(x),
\end{dcases}
\quad t \in (0,T), \, x \in (0,1),
\end{equation}
with $R_0=-\Lambda_+(0)Q_0 \Lambda_-(0)^{-1}$ and $R_1=-\Lambda_-(1)Q_1 \Lambda_+(1)^{-1}$, where $\Lambda_-=\diag(\lambda_1,\ldots,\lambda_m)$ and $\Lambda_+=\diag(\lambda_{m+1},\ldots,\lambda_{m+p})$.
We will disprove the above observability inequality if one of the matrices $Q_0$ or $Q_1$ is not full row rank.
We consider only the case $\rank Q_0<p$, the other one being similar.
To disprove inequality \eqref{obs ineq}, we define a sequence of final data $(z^{1,\nu})_{\nu \geq 1}$ as follows.
Since $\rank Q_0<p$, there exists a nonzero $\eta \in \R^p$ with
$$R_0^* \eta=0.$$
We then take $z^{1,\nu}_-=0$ and, for every $1 \leq j \leq p$,
$$
z^{1,\nu}_{m+j}(x)=
\begin{dcases}
g^\nu(\phi_{m+j}(x)) \eta_j & \text{ if } 0<x<\phi_{m+j}^{-1}(T_n), \\
0 & \text{ otherwise, }
\end{dcases}
$$
where $g^\nu:[0,T_n] \to \R$ will be determined below and
$$
\phi_{m+j}(x)=\int_0^x \frac{1}{\lambda_{m+j}(\xi)} \, d\xi.
$$
Note that $T_n \leq \phi_{m+j}(1)=T_{m+j}$ (see Section \ref{sect prob desc}).
For clarity, we will temporarily drop the dependence in $\nu$ below.

\item
We will first establish that
$$z_-=0.$$
To this end, it will be convenient to consider negative values for the time parameter $t$ (note that, in fact, the solution to the adjoint system \eqref{adj syst} belongs to $C^0((-\infty,T];L^2(0,1)^n)$).
Following the characteristics, it is equivalent to show that $z_-(\cdot,1)=0$.
We will argue by induction and prove that, for every $k \in \N$, we have
\begin{equation}\label{hyp induc}
z_i(t,1)=0, \quad \modif{\text{ for a.e. }} T-T_i-k(T_1+T_n)<t<T,
\end{equation}
for every $1 \leq i \leq m$.
For $k=0$, this is satisfied by very definition of $z^1_-=0$.
Assume now that this holds for some $k \geq 0$.
We first note that, by very definition of $z^1_+$, we have
\begin{align}
z_+(t,0) &=g(T-t) \eta, & & \modif{\text{ for a.e. }} T-T_n<t<T, & \nonumber
\\
z_{m+j}(t,0) &=0, & & \modif{\text{ for a.e. }} T-T_{m+j}<t<T-T_n, & \label{zmpj zero init}
\end{align}
for every $1 \leq j \leq p$.
The first condition and the boundary condition $z_-(t,0)=R_0^* z_+(t,0)$ then imply
\begin{equation}\label{zm zero first int}
z_-(t,0)=0, \quad \modif{\text{ for a.e. }} T-T_n<t<T.
\end{equation}
After these observations, let us now prove \eqref{hyp induc} with $k+1$ instead of $k$.
It follows in particular from our assumption that
$$z_-(t,1)=0, \quad \modif{\text{ for a.e. }} T-T_1-k(T_1+T_n)<t<T.$$
Using the boundary condition $z_+(t,1)=R_1^* z_-(t,1)$, we deduce that
$$z_+(t,1)=0, \quad \modif{\text{ for a.e. }} T-T_1-k(T_1+T_n)<t<T.$$
Following the characteristics of $z_+$, we obtain
$$z_{m+j}(t,0)=0, \quad \modif{\text{ for a.e. }} T-T_1-k(T_1+T_n)-T_{m+j}<t<T-T_{m+j},$$
for every $1 \leq j \leq p$.
Combined with \eqref{zmpj zero init}, this gives in particular
$$
z_+(t,0)=0, \quad \modif{\text{ for a.e. }} T-T_1-k(T_1+T_n)-T_n<t<T-T_n.
$$
The boundary condition $z_-(t,0)=R_0^* z_+(t,0)$ and \eqref{zm zero first int} then give
$$
z_-(t,0)=0, \quad \modif{\text{ for a.e. }} T-(k+1)(T_1+T_n)<t<T.
$$
Following the characteristics of $z_-$ until they touch $x=1$, and using \eqref{hyp induc} for $k=0$, we obtain \eqref{hyp induc} with $k+1$ instead of $k$.
This proves the induction.

\item
Since now $z_-=0$, we only have to consider the system satisfied by $z_+$.
The $j$-th component of the solution $z_+$ is also now given by
$$
z_{m+j}(t,x)=
\begin{dcases}
z^1_{m+j}(\phi_{m+j}^{-1}(T-t+\phi_{m+j}(x))) & \text{ if } t-\phi_{m+j}(x)+T_{m+j}>T, \\
0 & \text{ if } t-\phi_{m+j}(x)+T_{m+j}<T.
\end{dcases}
$$

Using that $z^1_{m+j}(x)$ is zero for $x>\phi_{m+j}^{-1}(T_n)$, we have
$$
\norm{z^1_+}_{L^2(0,1)^p}^2
=
\sum_{j=1}^p \int_0^{\phi_{m+j}^{-1}(T_n)} \abs{z^1_{m+j}(x)}^2 \, dx,
$$
and we can compute (using that $T \geq T_n$)
$$
\int_0^T \norm{z_+(t,\cdot)}_{L^2(0,1)^p}^2 \, dt=
\int_0^{T_n} \sum_{j=1}^p \int_{\phi_{m+j}^{-1}(s)}^{\phi_{m+j}^{-1}(T_n)} \abs{z^1_{m+j}(\xi)}^2 \frac{\lambda_{m+j}(\phi_{m+j}^{-1}(\phi_{m+j}(\xi)-s))}{\lambda_{m+j}(\xi)} \, d\xi ds.
$$
Let us introduce $f:[0,T_n] \to \R$ defined by
\begin{equation}\label{rela f z1}
f(s)=\sum_{j=1}^p \int_{\phi_{m+j}^{-1}(s)}^{\phi_{m+j}^{-1}(T_n)} \abs{z^1_{m+j}(\xi)}^2 \, d\xi.
\end{equation}
Then, the observability inequality \eqref{obs ineq}, the fact that $z_-=0$ and the previous computations, imply that
$$
f(0) \leq CL \int_0^{T_n} f(s) \, ds,
$$
for some $L>0$ depending only on $\Lambda$.
It is clear that such an inequality cannot be true for all nonnegative smooth functions $f:[0,T_n] \to \R$ with $f(T_n)=0$ and $f'<0$.
For instance, it is violated by the sequence $(f^\nu)_{\nu \geq 1}$ defined by
$$
f^\nu(s)=\left(\frac{T_n-s}{T_n}\right)^\nu, \quad 0 \leq s \leq T_n.
$$
To get to the conclusion, we simply ``invert'' \eqref{rela f z1} for such a sequence: we define $g^\nu$ by
$$
g^\nu(t)=
\left(-(f^{\nu})'(t)
\sum_{j \in J}
\frac{1}{\abs{\eta_j}^2 \lambda_{m+j}(\phi_{m+j}^{-1}(t))}
\right)^{\frac{1}{2}},
$$
where $J=\ens{j \st \eta_j \neq 0}$ is nonempty by definition of $\eta$.
Note that $g^\nu \in L^2(0,T_n)$.

\end{enumerate}
\end{proof}

\begin{remark}
Let us mention that the exact controllability of general systems with the identity as control operator has been studied in the literature.
Notably, it was proved in the interesting work \cite{Zwa13} that, in a Hilbert space framework, for such a system to be exactly controllable in some time it is necessary and sufficient that its generator $A$ is the extension of an operator $\bar{A}$ such that $-\bar{A}$ generates a $C_0$-semigroup.
\modif{This condition is not always easy to check in practice though, and this is why we have instead chosen to disprove the observability inequality.}
\end{remark}

\section*{Acknowledgments}

This project was supported by National Natural Science Foundation of China (Nos. 12122110, 12071258 and 12471421) and National Science Centre, Poland UMO-2020/39/D/ST1/01136.
For the purpose of Open Access, the authors have applied a CC-BY public copyright licence to any Author Accepted Manuscript (AAM) version arising from this submission.

%\renewcommand*{\bibfont}{\small}
%\printbibliography

%\footnotesize
\bibliographystyle{amsalpha}
\bibliography{biblio}

\end{document}